\documentclass{article}

\usepackage[centertags]{amsmath}
\usepackage{hyperref}
\usepackage{amsfonts}
\usepackage{amssymb}
\usepackage{amsthm}
\usepackage{newlfont}
\usepackage{mathtools}
\usepackage[top=1in, bottom=1in, left=1in, right=1in]{geometry}

\theoremstyle{plain}
\newtheorem{theorem}{Theorem}[section]

\newtheorem{corollary}[theorem]{Corollary}
\newtheorem{lemma}[theorem]{Lemma}

\theoremstyle{definition}
\newtheorem{definition}[theorem]{Definition}
\newtheorem{rem}[theorem]{Remark}
\newtheorem{claim}[theorem]{Claim}

\newenvironment{eq}[1]{\begin{equation} \label{#1}}{\end{equation}}

\newcommand{\R}{\mathbb{R}}

\newcommand{\Z}{\mathbb{Z}}
\newcommand{\C}{\mathbb{C}}
\newcommand{\Q}{\mathbb{Q}}
\newcommand{\HH}{\mathit{HH}}
\newcommand{\QQ}{\mathcal{Q}}
\newcommand{\Hidden}[1]{}

\newcommand{\la}{\lambda}
\newcommand{\1}{\mathbf 1}

\DeclareMathOperator{\Det}{Det}
\DeclareMathOperator{\Tr}{Tr}

\begin{document}

\title{Perfect state transfer on graphs with a potential}
\author{Mark Kempton\footnote{Center of Mathematical Sciences and Applications, Harvard University, Cambridge MA, mkempton@cmsa.fas.harvard.edu}~~~~~Gabor Lippner\footnote{Department of Mathematics, Northeastern University, Boston MA, g.lippner@neu.edu}~~~~~Shing-Tung Yau\footnote{Department of Mathematics, Harvard University, Cambridge MA, yau@math.harvard.edu}}
\date{}

\maketitle

\begin{abstract}
In this paper we study quantum state transfer (also called quantum tunneling) on graphs when there is a potential function on the vertex set. We present two main results.  First, we show that for paths of length greater than three, there is no
potential on the vertices of the path for which perfect state transfer between the endpoints can occur.  In particular, this answers a question raised by Godsil in Section 20 of~\cite{godsil}.  Second, we show that if a graph has two vertices that share a common neighborhood, then there is a potential on the vertex set for which perfect state transfer will occur between those two vertices.  This gives numerous examples where perfect state transfer does not occur without the potential, but adding a potential makes perfect state transfer possible.  In addition, we investigate perfect state transfer on graph products, which gives further examples where perfect state transfer can occur.
\end{abstract}

\section{Introduction}
Given a graph $G$, the \emph{discrete Schr\"odinger equation} on $G$ is given by
\begin{eq}{eq:shrod}
\frac{d}{dt}\varphi_t = iH\varphi_t
\end{eq}
where $\varphi_t:V(G)\rightarrow \C$ is a function on the vertex set of $G$, and $H$ is the graph Hamiltonian. Equation (\ref{eq:shrod}) describes the evolution of the quantum state of a particle on the graph $G$ with time.  In this paper, we take $H = A-Q$ where $A$ is the adjacency of $G$, and $Q$ is a diagonal matrix whose entries represent energy at each vertex.  The matrix $Q$ is called a potential on the graph.  It is also common to take $H = \Delta - Q$ where $\Delta$ is the graph Laplacian, but since we are allowing $Q$ to be a general diagonal matrix, our results will still apply to this case.

We will be studying solutions to $(\ref{eq:shrod})$ for which $\varphi_0$ is a characteristic function for a single vertex ($\varphi_0(x) = 1$ if $x=u$ and 0 otherwise).  That is, the quantum state of the particle completely concentrated at a single vertex.
\begin{definition}
We say that there is \emph{perfect state transfer} from vertex $u$ to vertex $v$ if there is some time $T$ at which the solution to (\ref{eq:shrod}) satisfies $|\varphi_T(v)| = 1$ and $\varphi_T(x) = 0$ for $x\neq v$.  
\end{definition}
In other words, perfect state transfer occurs when a particle starts at some specific vertex $u$, and after quantum evolution for some time $T$, the quantum state of the particle is completely concentrated on a single vertex $v$.  

\begin{rem}
Equation~\eqref{eq:shrod} also arises naturally in a well-studied model of quantum communication. In this model the nodes of the graph represents a collection of spin-1/2 particles (qubits), and edges represent couplings between particles. Each particle has a ground state and an excited state, so the whole system is modeled on $(\C^2)^{\otimes n}$. The system's evolution is described by its Hamiltonian
\[ H_{XX} = \frac{1}{2}\sum_{ (i,j) \in E(G)} X_i X_j + Y_i Y_j + \sum_{i \in V(G)} Q_i \cdot Z_i ,\] where $X_i, Y_i, Z_i$ are the standard Pauli matrices. The restriction of this system to the single-excitation subspace, leads to equation \eqref{eq:shrod}. 
Perfect state transfer in this model corresponds to a starting state, where only qubit $u$ is excited, evolving to a terminal state where only qubit $v$ is excited. Such an evolution has  applications to quantum communication by providing spin networks architectures through which quantum information can be losslessly transmitted. See \cite{kay2010} for details, where certain constructions for networks with edge-weights are described that have perfect state transfer. In this paper we focus on networks with uniform couplings - that is where all the edge weights are equal and thus can be chosen to be 1.
\end{rem}

There is a rapidly growing literature studying perfect state transfer on graphs when there is no potential, i.e., in the case where $Q=0$, so that $H = A$.  See, for example, \cite{Bose2003,Cheung2011,Christandl2004,Ge2011,godsil0, godsil, godsil2} for several results in this case.  It is apparent from this literature that perfect state transfer is a rather rare phenomenon, and constructing examples when this occurs can be quite difficult.  Our goal is to determine when adding a potential to the vertices can make perfect state transfer occur.  This is also investigated in \cite{Casaccino2009} (where the potential is referred to as an \emph{energy shift}).  While we are far from a complete answer to this question, we do describe an infinite family of graphs where adding a potential makes perfect state transfer possible (and for many graphs within this family, perfect state transfer does not occur without the potential), and we show that there are graphs where adding a potential cannot help.  


\paragraph{Paths:}
It has been shown in \cite{godsil} that perfect state transfer does not occur between any vertices of any path graph, except for the endpoints of a path on two vertices or of a path on three vertices. It has been conjectured to be the case in~\cite{Casaccino2009}, and subsequently raised as a question in~\cite{godsil}, whether perfect state transfer could be induced between the endpoints of a path of arbitrary length by placing a suitable potential on the endpoints. This conjecture in~\cite{Casaccino2009} was based on numerical experiments. The authors gave specific values of potential and corresponding times at which the strength of state transfer becomes very close to perfect. Results in \cite{tunneling} implied that this conjecture is true asymptotically (as the value of potential at each endpoint goes to infinity, the probability that tunneling occurs at some time goes to 1), and later an approximate version of the conjecture was confirmed in~\cite{stolze2012}. In fact, \cite{Casaccino2009} formulated an even bolder conjecture: that any network can support perfect state transfer under a suitable potential. 

Our first main result disproves both versions of the conjecture from~\cite{Casaccino2009}.  In fact, we show in general that given any potential on a path of length 4 or more, there is no time at which perfect state transfer occurs between the endpoints.

\begin{theorem}\label{thm:main1} There is no potential $Q$ that induces perfect state transfer between the endpoints of a path of length at least 4.
\end{theorem} 

A spectral characterization for perfect state transfer has long been known (see e.g.~\cite{kay2010}).  It consists of a symmetry condition on the eigenvectors and a rationality and parity condition on the eigenvalues (see Lemma \ref{lem:eig} and Corollary \ref{cor:rational} below, where we reproduce the characterization for the readers convenience). 

The symmetry and rationality conditions have been used throughout the literature to prove impossibility of perfect state transfer, but, to our knowledge, the parity condition hasn't yet been put to work. Our main contribution is exploiting the parity condition in a meaningful way - through a mod 2 reduction of matrices and eigenvalues.
 
Theorem~\ref{thm:main1} is also interesting in light of the result in~\cite{friedland} which states that any $n$ distinct real numbers can be achieved as the spectrum of a symmetric Jacobi (tridiagonal) matrix.  A natural question to ask to extend this result is if it is possible to achieve any spectrum by varying only the main diagonal entries, and holding all other entries fixed.  Our result can then be interpreted as saying that there are sets of eigenvalues (particularly, those needed as a necessary condition for perfect state transfer) that a Jacobi matrix cannot achieve if we fix the off-diagonal entries at 1.  In other words, varying only the diagonal entries of a symmetric Jacobi matrix does not give enough freedom to achieve any possible set of eigenvalues.

\paragraph{Nodes with identical neighborhoods:}
Our other main result is that if a graph has two vertices, $u$ and $v$, that share the same neighborhood, then there exists a potential on the vertex set of the graph for which perfect state transfer will occur between $u$ and $v$ at some time.  This provides infinitely many graphs where perfect state transfer can occur with some potential.  For many of these graphs, perfect state transfer does not occur without potential (consider stars, for example).  Our proof again uses a rationality condition on the eigenvalues.  In this case, we prove that having freedom on the diagonal gives enough genericity to achieve the eigenvalues required.  In \cite{Casaccino2009}, it is shown that for the complete graph, and the complete graph missing one edge, there is a choice of a potential on the vertices that induces perfect state transfer.  Our result is a far reaching generalization of both of these cases. 

\begin{theorem}\label{thm:nbhd}
Let $G$ be a graph on $n+2$ vertices with two vertices $u$ and $v$ that share the same neighbors, and such that $u\not\sim v$. There is a potential $Q: V(G) \to \R$ for which there is perfect state transfer between $u$ and $v$.
\end{theorem}

Theorem~\ref{thm:nbhd} is also interesting in light of the result of \cite{godsil2}, which states that, without potential, there are only finitely many connected graphs of maximum degree $k$ in which perfect state transfer can occur, for any given $k$. Our result shows that this is not true when we allow a potential.  Indeed, given any maximum degree $k$, simply choose any graph on any number of vertices with maximum degree $k-2$, and then add two vertices, attaching each of them to the same vertices (attaching them to at most $k$ vertices).  Then the resulting graph has maximum degree $k$, and perfect state transfer occurs in the graph.  Clearly there are infinitely many connected graphs for which we can do this.

\paragraph{Graph products:} In our final section, we will show that if we have two graphs in which perfect state transfer occurs at the same time with some potential, then there is a potential that we can put on the cartesian product of the two graphs for which perfect state transfer will also occur.  From this, we can construct more examples of graphs with potential where perfect state transfer occurs.  In particular, taking products can produce such examples that do not satisfy this condition that two vertices have the same neighborhood.
\section{Preliminaries}

Given a graph $G$ with $n$ vertices let $H = A - Q$ denote the graph Hamiltonian, where $A$ is the adjacency matrix, and $Q=diag(Q_1,\cdots,Q_n)$ a diagonal matrix with real entries.  Let $\varphi_0:V(G) \rightarrow \C$ be a complex-valued function on the vertex set of $G$ satisfying $||\varphi_0||_2 = 1$.  In \cite{Casaccino2009}, it is described how the adjacency matrix describes the ``XY" interaction of $n$ spin-1/2 quantum particles.  Using the adjacency matrix specifically gives the ``XY" Hamiltonian.  Other Hamiltonians can be given, particularly by using the graph Laplacian.

Define 
\[ \varphi_t(x) = e^{itH}\varphi_0(x)\]
and observe that $\varphi_t$ is a solution of (\ref{eq:shrod}).  We will denote $U(t) = e^{itH}$.  Note that the exponential of the matrix is given by 
\[
U(t) = e^{itH} = \sum_\lambda e^{it\lambda}xx^T
\]
where the sum is taken over eigenvalues $\lambda$ of $H$ and $v$ is the corresponding unit eigenvector.  In particular, note that 
\begin{eq}{eq:u0}
I = U(0) = \sum_\lambda xx^T.
\end{eq}
In addition, since $H$ is symmetric, each $\lambda$ above is real, and each $x$ can be assumed to have all real entries.  Observe also that $U(t)$ is a unitary matrix for all $t$, and therefore $||\varphi_t||^2 =  1$ for all $t$. If we let $\1_u$ denote the indicator vector for vertex $u$, then it is evident that perfect state transfer from $u$ to $v$, defined above, occurs at time $T$ if
\[
U(T)\1_u = \gamma\1_v
\]
for some $\gamma\in\C$ with $|\gamma|=1$. Then clearly, perfect state transfer from $u$ to $v$ occurs at time $T$ if and only if
\[|U(T)_{u,v}| = 1.\]

Versions of the following lemma and corollary are used throughout the literature on perfect state transfer.  See for example \cite{godsil0}.  We give a proof for completeness.

\begin{lemma}\label{lem:eig}
Let $u,v$ be vertices of $G$, and $H$ the Hamiltonian.  Then perfect state transfer from $u$ to $v$ occurs at some time if and only if the following two conditions are satisfied:
\begin{enumerate}
\item Every eigenvector $x$ of $H$ satisfies either $x(u) = x(v)$ or $x(u) = -x(v)$.
\item If $\{\lambda_i\}$ are the eigenvalues for eigenvectors with $x(u) = x(v)$, and $\{\mu_j\}$ are the eigenvalues for the eigenvectors with $x(u) = -x(v)$, and $x(u)$ and $x(v)$ are non-zero, then there exists some time $T$ such that
\[
e^{iT\lambda_1} = e^{iT\lambda_i} = -e^{iT\mu_j}
\]
for all $i,j$.
\end{enumerate}
\end{lemma}
\begin{proof}
Let $\{x_i\}$ be a set of orthonormal eigenvectors of $H$, and $\{\lambda_i\}$ the corresponding eigenvalues.  Then we can write
\[
U(t) = \sum_{i=1}^n e^{it\lambda_i}x_ix_i^T.
\]

($\impliedby$) Since
\[
U(t)_{u,v} = \sum_{i=1}^n e^{it\lambda_i}x_i(u)x_i(v)
\]
then if conditions 1 and 2 are satisfied, we get
\[
U(T)_{u,v} = e^{iT\lambda_1}\sum_{i=1}^n x_i(u)^2 = e^{iT\lambda_1}
\]
by (\ref{eq:u0}).  Therefore $|U(T)_{u,v}|=1$, so perfect state transfer occurs from $u$ to $v$ at time $T$. 

($\implies$) Assuming perfect state transfer occurs from $u$ to $v$ at time $T$, we have
$U(T)\1_u = \gamma\1_v$ and hence $x_k^TU(T)\1_u = \gamma x_k^T\1_v$ for any $k$, and therefore, by the above, since the $x_i$ are orthonormal,
\[
e^{iT\lambda_k}x_k^T\1_u = \gamma x_k^T\1_v
\]
which implies that
\[
x_k(u) = e^{-iT\lambda_k}\gamma x_k(v).
\]
Since $|e^{-iT\lambda_k}\gamma| = 1$, and the $x_i$ are real vectors, we see that $e^{-iT\lambda_k}\gamma = \pm1$, which gives condition 1.  

Now, relabel the eigenvalues so that $\lambda_i$ and $\mu_j$ are as in the statement of condition 2.  Then we have
\[
U(T)_{u,v} = \sum_{\lambda_i}e^{iT\lambda_i}x_i(u)^2 - \sum_{\mu_j}e^{iT\mu_j}x_j(u)^2.
\]
By assumption, we have that $|U(T)_{u,v}|=1$. By (\ref{eq:u0}), we have that $\sum x_i(u)^2 = 1$, then the only way for this sum to have modulus 1 is for the phase of each of the $e^{iT\lambda_i}$ for which $x_i(u)\neq0$ to line up to point in the same direction, and each of the and $e^{iT\mu_j}$ for which $x_j(u)\neq0$ to line up to point in the opposite direction.  This gives condition 2.
 
\end{proof}
\begin{corollary}\label{cor:rational}
Using the notation of Lemma \ref{lem:eig}, if perfect state transfer occurs from $u$ to $v$, then
\[
\frac{\lambda_i - \lambda_j}{\lambda_k - \lambda_\ell} \in \Q
\]
and 
\[
\frac{\lambda_i - \mu_j}{\lambda_k - \lambda_\ell} = \frac{odd}{even}
\]
for all $i,j,k,\ell$.
\end{corollary}
Here, $odd/even$ is used to indicate a rational number whose numerator is odd and denominator is even.  Also, note that the statement of the corollary with the role of $\mu$ and $\lambda$ reversed is clearly true as well.

One further observation to make here is that adding any multiple of the identity to $H$ simply shifts each eigenvalue by the same amount, keeping the same eigenvectors.  Therefore this does not affect the conditions of Lemma \ref{lem:eig}.  Therefore, the potential can be scaled by any constant shift at each vertex, and this will not affect whether or not perfect state transfer is possible, or the time at which it occurs.

\section{Potential on paths}

In this section, we will be investigating state transfer on paths.  It is known from \cite{godsil} that without a potential, perfect state transfer can only occur on a path of length 2 or 3, and not for longer paths.  As mentioned in the introduction, it was conjectured in \cite{godsil} that potential could be put on the endpoints of longer paths to make perfect state transfer occur.  We will show that this is false, and indeed, perfect state transfer on a path of length greater than 3 is impossible with any  potential on the path.  

\begin{rem} If there is perfect state transfer on a path, then by Lemma 2 of~\cite{kay2010} the potential has to be symmetric around the center of the path. Henceforth we will restrict to such symmetric potentials.
\end{rem}

Our first result will be to completely characterize potentials on $P_3$, the path of length 3, for which perfect state transfer occurs.  We will then prove Theorem~\ref{thm:main1} in two parts (see Theorems~\ref{thm:evenpath} and~\ref{thm:oddpath}, treating odd and even length paths separately. 

We remark that for the path of length 2, $P_2$, it is easy to see that perfect state transfer will occur with any symmetric potential, since a symmetric potential is constant, so simply translates the potential by a constant amount.  It is also easy to see that if the potential is not symmetric (in this case, meaning the values on the two vertices are distinct), then the eigenvectors will not have the form required from Lemma \ref{lem:eig}, and so perfect state transfer is impossible.

We will begin with some observations concerning the characteristic polynomials of paths. Let 
\[p_n(x;Q_1,Q_2,...,Q_n) = \det (x I_n - H_n)\] be the characteristic polynomial of the Hamiltonian, $H_n = A_n +  \cdot diag(Q_1,Q_2,\dots,Q_n)$, of the $n$-vertex path with potential at each vertex. Let $L_{2n} = p_{2n}(x;Q_1,...Q_n,Q_n,...,Q_1)$ be the characteristic polynomial for the Hamiltonian with symmetric potential, and likewise, $L_{2n+1} = p_{2n+1}(x;Q_1,...,Q_n,Q_{n+1},Q_n,...,Q_1)$.

\begin{lemma}\label{lem:Ln}
We have the following identities for $p_n$ and $L_n$:
\begin{align}
p_n(x;Q_1,...,Q_n) &= (x-Q_n)\cdot p_{n-1}(x;Q_1,...,Q_{n-1}) - p_{n-2}(x;Q_1,...,Q_{n-2})\\
L_{2n} &= (p_n + p_{n-1})(p_n - p_{n-1}) \\ 
L_{2n+1} &= p_n(p_{n+1} - p_{n-1})
\end{align}

Furthermore, in each of the above factorizations, the roots of the first factor correspond to eigenvectors $f$ for which $f(1) = -f(n)\neq0$, and the roots in the second factor correspond to eigenvectors $f$ with $f(1) = f(n)\neq0$.  
\end{lemma}

\begin{proof}
The first follows from direct expansion of the determinant, expanding along the last row.

For the second, suppose we have
\begin{eq}{submatrix}
\begin{bmatrix}
Q_1&1&0& &&\\1&Q_2&1&\ &&\\0&1&Q_3&&&\\&&&\ddots&\\&&& &Q_{n-1}&1\\&&&&1&Q_{n}+1
\end{bmatrix}\begin{bmatrix}
a_1\\a_2\\a_3\\\vdots\\a_{n-1}\\a_n
\end{bmatrix} = \lambda\begin{bmatrix}
a_1\\a_2\\a_3\\\vdots\\a_{n-1}\\a_n
\end{bmatrix}.
\end{eq}
For ease of notation, let $a = (a_1,...,a_n)^T$ and we will use $\bar a$ to denote the ``reversal" of $a$, that is, $\bar a = (a_n,...,a_1)^T$.
Then direct computation shows that
\[
H_n\begin{bmatrix}a\\ \bar a
\end{bmatrix} = \lambda\begin{bmatrix}
a\\ \bar a
\end{bmatrix}.
\]
In addition, the characteristic equation of the matrix on the left in (\ref{submatrix}) is 
$p_n - p_{n-1}$.

In a similar manner, if we have
\begin{eq}{submatrix2}
\begin{bmatrix}
Q_1&1&0& &&\\1&Q_2&1&\ &&\\0&1&Q_3&&&\\&&&\ddots&\\&&& &Q_{n-1}&1\\&&&&1&Q_n-1
\end{bmatrix}\begin{bmatrix}
a_1\\a_2\\a_3\\\vdots\\a_{n-1}\\a_n
\end{bmatrix} = \lambda\begin{bmatrix}
a_1\\a_2\\a_3\\\vdots\\a_{n-1}\\a_n
\end{bmatrix}
\end{eq}
then
\[
H_n\begin{bmatrix}
a\\-\bar a
\end{bmatrix} = \lambda\begin{bmatrix}
a\\-\bar a
\end{bmatrix}.
\]
and the characteristic equation of the matrix on the left in (\ref{submatrix2}) is 
$p_n + p_{n-1}$. This gives the second equation of the lemma.

For $L_{2n+1}$, if we have 
\begin{eq}{submatrix3}
\begin{bmatrix}
Q_1&1&0& &&\\1&Q_2&1&\ &&\\0&1&Q_3&&&\\&&&\ddots&\\&&& &Q_n&1\\&&&&2&Q_{n+1}
\end{bmatrix}\begin{bmatrix}
a_1\\a_2\\a_3\\\vdots\\a_{n-1}\\a_n\\a_{n+1}
\end{bmatrix} = \lambda\begin{bmatrix}
a_1\\a_2\\a_3\\\vdots\\a_{n-1}\\a_n\\a_{n+1}
\end{bmatrix}
\end{eq}
then 
\[
H_n\begin{bmatrix}
a\\a_{n+1}\\ \bar a
\end{bmatrix} = \lambda\begin{bmatrix}
a\\a_{n+1}\\ \bar a
\end{bmatrix}.
\]
The characteristic equation of the matrix on the left in (\ref{submatrix3}) is 
$p_{n+1} - p_{n-1}$.

Finally, if we have
\begin{eq}{submatrix4}
\begin{bmatrix}
Q_1&1&0& &&\\1&Q_2&1&\ &&\\0&1&Q_3&&&\\&&&\ddots&\\&&& &Q_{n-1}&1\\&&&&1&Q_n
\end{bmatrix}\begin{bmatrix}
a_1\\a_2\\a_3\\\vdots\\a_{n-1}\\a_n
\end{bmatrix} = \lambda\begin{bmatrix}
a_1\\a_2\\a_3\\\vdots\\a_{n-1}\\a_n
\end{bmatrix}
\end{eq}
then
\[
H_n\begin{bmatrix}
a\\0\\-\bar a
\end{bmatrix} = \lambda\begin{bmatrix}
a\\0\\-\bar a
\end{bmatrix}.
\]
and the characteristic equation of the matrix on the left in (\ref{submatrix4}) is 
$p_n$.

Finally, it is clear from above that the eigenvectors we found satisfy $f(1)=f(n)$ or $f(1)=-f(n)$ respectively.  Furthermore, straightforward computation shows that if $f(1)=0$, then $f(x)=x$ for all $x$, thus $f(1)$ and $f(n)$ are non-zero.
This gives the lemma.

\end{proof}

\subsection{State transfer on $P_3$.}

In this section we will characterize every possible potential on $P_3$ for which perfect state transfer can occur, up to scaling by an additive constant.   Note that it is clear that if we have an asymmetric potential on $P_3$, then the eigenvectors will not have the required form of Lemma \ref{lem:eig}, so we assume that we have equal potential at each endpoint.   By adding a multiple of the identity, we can assume, without loss of generality, that this value is 0, so we will simply assume the potential on the midpoint is $Q$.

\begin{theorem}
Let $G = P_3$, the path on 3 vertices, let $u,v$ be the endpoints of the path, and let $Q = diag(0,q,0)$.  Then there is perfect state transfer from $u$ to $v$ if and only if there exist integers $k$ and $\ell$ of opposite parity such that
\[
(k^2-\ell^2)q^2 = 8\ell^2.
\]
When this is the case, perfect state transfer occurs at time
\[
t =\frac{2\pi k}{\sqrt{q^2+8}}.
\]
\end{theorem}
\begin{proof}
We have
\[
H = \begin{bmatrix}
0&1&0\\1&q&1\\0&1&0\end{bmatrix}.
\]
Observe that 
\[
\begin{bmatrix}0&1&0\\1&q&1\\0&1&0\end{bmatrix}\begin{bmatrix}
1\\0\\-1\end{bmatrix} = \begin{bmatrix}
0\\0\\0
\end{bmatrix}.
\]
Let 
\[
\mu = \frac{q\pm\sqrt{q^2+8}}{2}
\]
and observe that
\[
\begin{bmatrix}0&1&0\\1&q&1\\0&1&0\end{bmatrix}\begin{bmatrix}
1\\ \mu\\1\end{bmatrix} = \mu\begin{bmatrix}
1\\ \mu\\1
\end{bmatrix}.
\]
We have thus found all three eigenvalues of $H$, with 0 corresponding to an eigenvector whose entries at the endpoints are opposite, and the others whose eigenvectors are constant on the endpoints.  Then by Lemma \ref{lem:eig}, perfect state transfer between $u$ and $v$ happens if and only if
\[
e^{it\sqrt{q^2+8}/2} = e^{-it\sqrt{q^2+8}/2} = -e^{-itq/2}
\]
The first equality implies that we must have
\[
t = \frac{2\pi k}{\sqrt{q^2+8}}
\]
for some integer $k$, which further implies, from the second equality, the following two cases.  If $k$ is even, then $kq/\sqrt{q^2+8}$ must be an odd integer, and if $k$ is odd, then $kq/\sqrt{q^2+8}$ must be an even integer.  Therefore, if there is an integer $\ell$ of opposite parity to $k$ such that
\[
(k^2-\ell^2)q^2 = 8\ell^2
\]
then we have perfect state transfer at time
\[
t = \frac{2\pi k}{\sqrt{q^2+8}}.
\]
\end{proof}

We remark that given any integer $\ell$, then any choice of $k>\ell$ of opposite parity yields a value of the potential $q$ for which perfect state transfer occurs, namely
\[
q = \sqrt{\frac{8\ell^2}{k^2-\ell^2}}.
\]
At this value of $q$, we see the value of $t$ from above at which state transfer occurs becomes
\[
t = \frac{2\pi}{\sqrt{8}}\sqrt{k^2 -\ell^2}.
\]
This has a remarkable consequence. Unless $q=0$, we see that 
\[ q\cdot t \geq 2\pi \sqrt{\ell^2} \geq 2 \pi,\] in other words, $q$ and $t$ cannot be small at the same time. Small values of the potential require long waiting times before tunneling first occurs. It is an interesting open question if this is a phenomenon for general graphs with potential for which tunneling occurs.  

We remark that the relationship between $q$ and $t$ seen above is consistent with with the relationship expected based on numerical evidence given in \cite{Casaccino2009}.

\subsection{Even length paths}

\begin{theorem}\label{thm:evenpath}
Let $G = P_{2n}$, the path on $2n \geq 4$ vertices, let $u,v$ be the endpoints of the path, and let $Q = diag(Q_1,\dots,Q_n,Q_n,\dots,Q_1)$ be any symmetric potential. Then perfect state transfer cannot occur from $u$ to $v$ for any values of the $Q_i$'s.
\end{theorem}

\begin{proof}
Let $H_+ = A_n+diag(Q_1,\dots,Q_{n-1},Q_n+1)$ and $H_- = A_n + diag(Q_1,\dots,Q_{n-1},Q_n-1)$. Let $\la_i : i=1,\dots,n$ and $\mu_j : j= 1,\dots,n$ denote the eigenvalues of $H_+$ and $H_-$ respectively. By Lemma~\ref{lem:Ln} we know that these correspond to symmetric and anti-symmetric eigenvectors. 

Assume there is perfect state transfer at some time $t$. Then by Corollary \ref{cor:rational}, the ratios of type $(\la_i - \mu_j)/(\la_k - \la_l)$ have to be equal to some odd/even fraction.

\begin{claim}\label{claim:l1-l2}
$\la_i - \la_j, \mu_i-\mu_j, \la_i- \mu_j$ are all rational.
\end{claim}
\begin{proof}
By the ratio condition we know that $(\la_i - \mu_j)/(\la_1 - \la_2)$ has to be rational for any $i,j$. Then the sum
\[ \sum_{i,j=1}^n \frac{\la_i - \mu_j}{\la_1-\la_2} = \frac{n}{\la_1 - \la_2} \left( \sum_{i=1}^n \la_i - \sum_{j=1}^n \mu_j\right) = (\Tr H_+ - \Tr H_-)\frac{n}{\la_1 - \la_2} = \frac{2n}{\la_1-\la_2}\] is rational, and so $\la_1 - \la_2$ is also rational. From this, the rationality of all other such differences follow from the ratio condition.
\end{proof}

\begin{corollary}
We can assume that all eigenvalues of $H_{2n}$ are rational.
\end{corollary}

\begin{proof}

As an immediate consequence of Claim~\ref{claim:l1-l2}, there is a real number $\alpha$ such that $\la_i - \alpha$ and $\mu_j - \alpha$ are rational for all $i,j \in \{1,2,\dots,n\}$. Then the eigenvalues of the matrix $H'_{2n} = H_{2n} - \alpha  I_n$ are rational and satisfy the ratio-conditions. Furthermore the potential is still symmetric. So if there was tunneling for $H_{2n}$ then there is also tunneling for $H'_{2n}$.
\end{proof}

From now on we are going to use this assumption without further warning.

\begin{claim}\label{cl:qirat}
All the $Q_i$'s are also rational.
\end{claim}

\begin{proof}
We proceed by a descending induction from $Q_n$ to $Q_1$. Suppose we've already shown that $Q_n, Q_{n-1}, \dots, Q_{n-k+1}$ are rational. (Here we also include the case where nothing was shown yet as $k=0$.)
Let us look at 
\[ \Tr( H_+^{2k+2} - H_-^{2k+2}) = \sum_{i=1}^n \la_i^{2k+2} - \mu_i^{2k+2}  \in \Q,\]
and observe that $\Tr( H_+^{2k+2} - H_-^{2k+2}) = (4k+4)Q_{n-k} + $ an integer coefficient polynomial in $Q_{n-k+1},Q_{n-k+2}, \dots, Q_n$. Then, by induction, $Q_{n-k}$ has to be rational. 

To see why the trace expression is indeed what we claim it is, write the $j$th diagonal entry of $H_{\pm}^{2k+2}$ as sum of weighted cycles of length $2k+2$ starting and returning to the $j$th node of the path. For such a weighted cycle to have different weights in $H_+$ and $H_-$, the cycle has to pass through the loop on the $n$th node. Thus $j \geq n-k$. Thus $ \Tr( H_+^{2k+2} - H_-^{2k+2})$ will not depend on $Q_1, \dots, Q_{n-k-1}$. If such a cycle further passes through the loop edge on vertex $n-k$, then it has to consist of all edges between $n-k$ and $n$ exactly once in both directions as well as the two loops at $n-k$ and $n$ respectively. There are $2$ such cycles for any starting point $n-k \leq j \leq n$. These together contribute $(4k+4)Q_{n-k}$ to $\Tr( H_+^{2k+2} - H_-^{2k+2})$, and everything else as an integer-weighted linear combinations of monomials depending only on the higher $Q$'s. From this the observation follows.
\end{proof}

Introducing a further shift in the potential we can now assume that $Q_n = 0$.

Let us write $Q_i = a_i / K$ where $a_i, K \in \Z$ and $(K,a_1,a_2,\dots, a_n) = 1$. Then the rational numbers $K\cdot \la_i, K \cdot \mu_j$ are eigenvalues of the integer matrices $\HH_{\pm} = K \cdot H_{\pm}$, so they have to be integers themselves. Let us write $l_i = K \cdot \la_i$ and $m_j = K \cdot \mu_j$. 

By the ratio condition we see that for any $i_1, i_2,j_1,j_2$:
\[ l_{i_1} - l_{i_2} = \frac{even}{odd} l_{j_1}-l_{j_2} = even,\] hence all $l_i$'s have the same parity. Similarly all $m_j$'s have the same parity. 

First suppose all $l_i$'s and all $m_j$'s are odd. Then, using that $s^3 \equiv s $ mod 8 for odd $s$, we get mod 8 that 
\[ 2K = \Tr( \HH_+ - \HH_-) = \sum l_i - m_i \equiv \sum l_i^3 - m_i^3 = \Tr( \HH_+^3 - \HH_-^3) = 8K^3 \equiv 0,\] hence $2K$ is divisible by 8, thus $K$ is divisible by 4. But then $\prod l_i = \Det \HH_+$ is even, which contradicts that all the $l_i$'s are odd.

Next observe that 
\[ \prod l_i - \prod m_j = \Det \HH_+ - \Det \HH_- =( K  D_{n-1}-K^2 D_{n-2} )-(-K D_{n-1}-K^2 D_{n-2}) = 2K D_{n-1} \] is even, where $D_{j} = \Det(K(A_{j} + diag(Q_1,\dots,Q_{j})))$. This means that neither the $l_i$'s have to have the same parity as the $m_j$'s (remember they all have the same parity within each group) and thus they all have to be even. 

Finally observe that this implies that $K$ has to be even as well. Assume for a contradiction that $K$ is odd. Since all $l$'s and $m$'s are even, we get that $2K D_{n-1}$ and $2K^2 D_{n-2}$ are both divisible by $2^{n}$. If $K$ is odd and $n \geq 2$ then in fact $D_{n-1}$ and $D_{n-2}$ both have to be even. However from this we get by a  induction that any $D_{n-j}$ is even $(j=1,2,\dots, n)$, which is impossible as $D_0 = 1$. The induction follows from the recursive formula 
\[ D_{j+2} = a_{j+2} D_{j+1} - K^2 D_j.\] Hence if $D_{j+2}$ and $D_{j+1}$ are even, but $K$ is odd, then $D_j$ must also be even. 

So far we have shown that if there is tunneling then all the $l_i$'s have to be even and $K$ has to be even.  The final contradiction will arise from examining the characteristic polynomial $P(x) = \prod_{i=1}^n (x-l_i)$ of $\HH_+$. Considering $P(x)$ over the field $F_2$ it reduces to $x^n$ since all the $l_i$'s are even. At the same time if we first consider $\HH_+$ mod 2, we get the diagonal matrix with entries $a_1, a_2, \dots, a_{n-1}, 0$, so the roots of its characteristic polynomial (which has to coincide with $P(x)$ mod 2) are $a_1, a_2, \dots, a_{n-1},0$. In particular, because of unique factorization over $F_2$, it follows that all the $a_i$'s are even. This, however, contradicts that $(K,a_1,\dots,a_{n-1}) = 1$. 

\end{proof}

\subsection{Odd length paths}

\begin{theorem}\label{thm:oddpath}
Let $G = P_{2n+1}$, the path on $2n+1 \geq 5$ vertices, let $u,v$ be the endpoints of the path, and let $Q = diag(Q_1,\dots,Q_n,Q_{n+1},Q_n,\dots,Q_1)$ be any symmetric potential. Then perfect state transfer cannot occur from $u$ to $v$ for any value of the $Q_i$'s.
\end{theorem}

\begin{proof}
First of all, by shifting the potential, we can assume that $Q_{n+1}=0$. Let $H_s$ and $H_a$ denote the matrices on the left hand side of \eqref{submatrix3} and \eqref{submatrix4} respectively. That is, $H_a = A_{n+1} + diag(Q_1,\dots,Q_n,0)$, and $H_s$ is obtained from $A_{n} + diag(Q_1,\dots,Q_n)$ by replacing the 1 in the last row by 2. Let us denote by $\la_1,\dots, \la_{n+1}$ and $\mu_1,\dots, \mu_n$ the eigenvalues of $H_a$ and $H_s$ respectively. 

Assume there is perfect state transfer at some time $t$. Again by Corollary \ref{cor:rational}, the ratios of type $(\la_i - \mu_j)/(\la_k - \la_l)$ have to be equal to some odd/even fraction.

\begin{claim}\label{claim:l1-l2Q}
$\la_i - \la_j, \mu_i-\mu_j, \la_i- \mu_j$ are all rational multiples of $Q= Q_1+\dots+Q_n$.
\end{claim}
\begin{proof}
By the ratio condition we know that $(\la_i - \mu_j)/(\la_1 - \la_2)$ has to be rational for any $i,j$. Then the sum
\[ \sum_{i,j=1}^n \frac{\la_i - \mu_j}{\la_1-\la_2} = \frac{1}{\la_1 - \la_2} \left( n \sum_{i=1}^n \la_i - (n+1)\sum_{j=1}^n \mu_j\right) = \frac{1}{\la_1-\la_2}\left(n \Tr H_s -(n+1)\Tr H_a\right) = \frac{-Q}{\la_1 - \la_2}\] is rational, and so $\la_1 - \la_2$ is a rational multiple of $Q$. From this the rationality of all other such differences follow from the ratio condition.
\end{proof}

\begin{claim}
Each $\la_i$ and each $\mu_j$ is a rational multiple of $Q$.
\end{claim}
\begin{proof}
\[ Q = \Tr H_a = \sum_{i=1}^n \mu_i = n\mu_1 + \sum_{i=1}^n \mu_i - \mu_1 = n \mu_1 + r Q\] for some $r \in \Q$. Hence $\mu_1 = (1-r)/n \cdot Q$. The same for the rest of the eigenvalues follows from the previous claim.
\end{proof}

\begin{claim}
$Q^2$ is rational.
\end{claim}
\begin{proof}
The following sum is rational multiple of $Q^2$ by Claim~\ref{claim:l1-l2Q}:
\[ \sum_{i=1}^{n+1} \la_i^2 - \sum_{j=1}^n \mu_j^2 = \Tr H_s^2 - \Tr H_a^2 = \sum_{i=1}^n Q_i^2 + 2n+2 - \sum_{j=1}^n Q_i^2 - (2n-2) = 4.\] So $Q^2$ has to be rational itself.
\end{proof}

\begin{claim}
Each $Q_i$ is a rational multiple of $Q$. 
\end{claim}

\begin{proof}
This proceeds exactly as the proof of Claim~\ref{cl:qirat}. Successively considering $\Tr H_s^{2j+1} - \Tr H_a^{2j+1}$ for $j=1,2,\dots$ we find that a rational multiple of $Q$ is equal to a rational linear combination of terms already shown to be a rational multiple of $Q$ and a non-zero rational multiple of $Q_{n+1-j}$. This implies, by induction on $j$, that $Q_{n+1-j}$ has to be a rational multiple of $Q$ itself. 

Note: we exploit at each step that $Q^2$ is rational, hence any odd power of $Q$ is a rational multiple of $Q$.
\end{proof}

Since $Q^2 \in \Q$ and all $Q_i$s are rational multiples of $Q$, we can let $Q_i = a_i/K\cdot \sqrt{a/b}$ where $a, b \in \Z$ are square free coprime integers, and $K, a_i \in \Z$ such that $(K,a_1,\dots,a_n)=1$. It can be also assumed that $(K,a)=1$, since otherwise $K, a, b$ could be replaced by $K/p, a/p, p b$ for any prime $p | (K,a)$. 

Further let $l_i = K \sqrt{ab} \la_i \in \Q$ and $m_j = K \sqrt{ab} \mu_j \in \Q$. Then then $l_i$s are eigenvalues of the matrix $\HH_s = \sqrt{ab}K \cdot H_s$ and $m_j$s are the eigenvalues of $\HH_a = \sqrt{ab} K \cdot H_a$.  It is easy to see that the characteristic polynomials of both of these matrices are monic and have integer coefficients, so all $l_i$s and $m_j$s are in fact integers. Furthermore by the ratio condition 
\begin{eq}{eq:ratiocondOdd} l_i - l_j = (l_1-m_1)\frac{\mbox{even}}{\mbox{odd}} = \mbox{even},\end{eq} so all the $l_i$'s have the same parity, and similarly all the $m_j$'s have the same parity.

Let us further write $l_i = 2^{\alpha_i}(2s_i + 1)$ and $m_j = 2^{\beta_j}(2t_j+1)$ where $\alpha_i, \beta_j \in \Z_{\geq 0}, s_i, t_j \in \Z$. Let $A = \max\{\alpha_1,\dots, \alpha_{n+1}\}$  and $\beta = \max\{\beta_1, \dots, \beta_{n}\}$.

\begin{claim}~
\begin{enumerate}
\item If $A = B$ then every $\alpha_i = A$ and every $\beta_j = A$. 
\item If $A < B$ then every $\alpha_i = A$ and every $\beta_j \geq B \geq A+1$.
\item If $A > B$ then every $\beta_j = B$ and every $\alpha_j \geq A \geq B+1$.
\end{enumerate}
\end{claim}
\begin{proof}
This follows simply from the ratio condition \eqref{eq:ratiocondOdd}: We can assume that $\alpha_1 = A$ and $\beta_1 = B$. Applying \eqref{eq:ratiocondOdd} with $j=1$ we get that $2^{\min\{A,B\}+1}$ divides $l_i - l_j$, so $\alpha_i \geq \min\{A,B\}$ and if $A \geq B+1$ then actually $\alpha_i \geq B+1$. The reverse cases follow similarly. 
\end{proof}

\begin{claim}
In fact $A=B$ is impossible in the previous claim.
\end{claim}

\begin{proof}
Suppose $A=B$. Then
\[ 0 = \Tr \HH_s - \Tr \HH_a = \sum_1^{n+1} l_i - \sum_1^{n}m_j = 2^A \left(\sum_1^{n+1} 2s_i+1 - \sum_1^n 2t_j + 1 \right) = 2^A \cdot \mathrm{odd},\] a clear contradiction.
\end{proof}

\begin{claim}
All the $l_i$s are even.
\end{claim}
\begin{proof}
\[ \prod_{i=1}^{n+1}l_i = \Det(\HH_s) = 2K^2 a b D_{n-1}\] is even, where $D_k = \Det( K\sqrt{ab} \cdot A_k + a \cdot diag(a_1,\dots,a_k))$. Since all $l_i$s have the same parity, they must all be even.
\end{proof}

\begin{claim}
The characteristic polynomial of $\HH_s$ mod 2 is $x$ times the characteristic polynomial of $\HH_a$ mod 2. Hence all the $m_j$s are also even.
\end{claim}

\begin{proof}
The first part is obvious from expanding the determinant defining the characteristic polynomial of $\HH_s$. The second part follows from the unique factorization of polynomials mod 2.
\end{proof}

\begin{claim}
$K^2 a b$ must be even.
\end{claim}

\begin{proof}
Since all the $m_j$s are even, their product, $D_n$, is even. Assume $K^2 a b$ is odd. Then, since all the $l_j$s are even and their product is $2K^2 ab D_{n-1}$, it follows that $D_{n-1}$ is even. Then by induction all the $D_k$s are even: $D_{k+2} = a a_{k+2} D_{k+1} - K^2 a b D_{k}$. If $D_{k+2}$ and $D_{k+1}$ are even, then so is $D_k$. This implies that $D_0 = 1$ is also even, a contradiction.
\end{proof}

\begin{claim}
$K$ is odd, and hence exactly one of $a$ and $b$ are even.
\end{claim}

\begin{proof}
Suppose $K$ is even. Then by the $(K,a)=1$ assumption $a$ is odd. Then the characteristic polynomial of $\HH_a$ mod 2 is equal to the characteristic polynomial of its diagonal mod 2. (Since all terms involving off-diagonal elements will be even.) So the roots of the characteristic polynomial mod 2 are equal to the diagonal elements mod 2. This means, since all $m_j$s are even, that all $a a_i$s have to be even, implying that all $a_i$s have to be even. But this contradicts $(K,a_1,\dots,a_n)=1$. 
\end{proof}

At this point we have $2^n | \prod_1^n m_j = D_n$ and $2^{n+1}| \prod_1^{n+1} l_i = 2K^2 ab D_{n-1}$ and thus (since $a,b$ are square-free) $2^{n-1}|D_{n-1}$. 

Then we obtain recursively that $2^k | D_k$ from the formula
\[ D_{k+2} = a a_{k+2} D_{k+1} - K^2 a b D_{k},\] since the $K^2 ab$ can only absorb a singe factor of 2. However, if there is any $k \leq n-1$ for which $2^{k+1}|D_k$, then we have $2^{k+1}$ divides both $D_{k}$ and $D_{k+1}$ so $2^k | D_{k-1}$ and then inductively $2^{j+1} | D_j$ for any $j \leq k$. In particular $2 | D_0 =1$ which is a contradiction.

This implies that for all $k\leq n-2$ the values $a a_{k+2}$ must be odd, otherwise we would get $2^{k+1}|D_k$ and a contradiction. This means that in fact $a$ is odd and all $a_i$ has to be odd for $i\geq3$. But this yields a contradiction: again looking at the characteristic polynomial of $\HH_a$ mod 2: all its roots should be even, but it coincides with the characteristic polynomial of just the diagonal mod 2 which apparently has at least one odd root (since $n\geq2$ in the statement of the theorem). 
\end{proof}

\section{Vertices with identical neighborhoods}

Let $G$ be a graph on $n+2$ vertices with two vertices $u$ and $v$ that share the same neighbors, and such that $u\not\sim v$.  The goal of this section is to investigate tunneling from $u$ to $v$. In particular we show the following result.

{
\renewcommand{\thetheorem}{\ref{thm:nbhd}}
\begin{theorem}
There is a potential $Q: V(G) \to \R$ for which there is perfect state transfer between $u$ and $v$.
\end{theorem}
\addtocounter{theorem}{-1}
}

The strategy of our proof is a perturbation argument. We find a suitable initial choice of the potential and show that in its neighborhood there is a dense set of potentials satisfying the theorem. This is made possible because the number of parameters turns out to be the same as the number of conditions to be satisfied. 

\begin{proof}
By Lemma~\ref{lem:eig} we need to find a potential for which the Hamiltonian satisfies two conditions.  

First we show that if $Q(u)=Q(v)$ then the first condition is automatically satisfied. It is easy to see that for such a potential $H$ has an eigenvalue $\lambda_0=Q(u)=Q(v)$, such that the corresponding eigenvector $\phi_0$ satisfies $\phi_0(u) = -\phi_0(v)$ and $\phi_0(x) = 0$ for $x\neq u,v$.  By a diagonal shift, we can assume without loss of generality that $Q(u)=Q(v)=0$ so that $\lambda_0=0$. Let $\lambda_1 \leq \lambda_2 \leq \dots \leq \lambda_{n+1}$ denote the other eigenvalues, and $\phi_1,...,\phi_{n+1}$ the corresponding orthonormal set of eigenvectors of $H$.  Since each $\phi_j$ must be orthogonal to $\phi_0$ for $j=1,...,n+1$, then we immediately see that $\phi_j(u) = \phi_j(v)$ for $j=1,...,n+1$. 

\begin{definition}\label{def:good}
Let us say that a potential $Q$ is \emph{good} if $Q(u) = Q(v)=0$ and all the eigenvalues of $H$ are simple.  Let $\QQ$ denote the set of good potentials. It is clear that $\QQ$ can be viewed as an open subset of $\R^n$.
\end{definition}

 So far we have shown that if $Q\in \QQ$ then the first condition of Lemma~\ref{lem:eig} is satisfied. 

In particular we see that there is only a single eigenvalue of the second type, so to satisfy the second condition all we need to ensure is existence of a time $t$ such that $e^{i\lambda_j t} = -1$ for each $j$.  That is, we must have $t\lambda_j$ is an odd multiple of $\pi$ for each $j=2,...,n$.  

By Lemma~\ref{lem:ratios} (see below) there is a potential $Q_0$ such that the total derivative of the map $\Phi : \QQ \to \R^n$ defined by
\[\label{eq:phi} \Phi(Q) := \left(\frac{\lambda_1}{\lambda_{n+1}},\dots,\frac{\lambda_n}{\lambda_{n+1}}\right) \] is invertible at $Q_0$. Since the set of points in $\R^n$ whose coordinates are all odd/odd rational numbers is dense, the inverse function theorem guarantees the existence of a potential $Q\in \QQ$ close to $Q_0$ such that 
\[ \Phi(Q) = \left( \frac{2p_1+1}{2q+1}, \dots, \frac{2p_n +1}{2q+1}\right)\]  where $q, p_1,\dots,p_n \in \Z$. For this potential $t = \pi (2q+1)/\lambda_{n+1}$ is a good choice of $t$.
\end{proof}
 
\begin{lemma}\label{lem:ratios}
There is a potential $Q_0 \in \QQ$ such that the total derivative of the map $\Phi$ defined in \eqref{eq:phi} is invertible.
\end{lemma}

Before we prove this lemma, we need to make some preparations. Let us denote the vertices in $V(G)\setminus \{u,v\}$ by $1,2,\dots,n$. For any potential $Q \in \QQ$ let us write $Q_i = Q(i)$. We know that $0$ is an eigenvalue of the Hamiltonian corresponding to the potential $Q$, so we can write its characteristic polynomial as $2x \cdot F(x,Q)$ where $F$ is a polynomial in $x$ and the $Q_j$s. 
By some abuse of notation we are going to use the following shorthands for derivatives of $F$:  we let $F'(x,Q) = \partial_x F(x,Q)$ and $F_j(x,Q) = \partial_{Q_j}F(x,Q)$.  

\begin{lemma}\label{lem:dense}
The set $\QQ$ is non-empty.

\end{lemma}

\begin{proof}
Take the values of $Q$ to be large, distinct real numbers.  Then $H$ can be thought of as a small perturbation of a diagonal matrix, and the lemma becomes clear.

\end{proof}

The following claim is a variant of the implicit function theorem.

\begin{claim}\label{cl:implicit}
For $Q\in \QQ$ the eigenvalues $\lambda_i$ depend smoothly in the potential near $Q$, and we have \[\partial_j \lambda_i(Q) := \partial_{Q_j} \lambda_i(Q) = \frac{ F_j(\lambda_i(Q), Q)}{F'(\lambda_i(Q),Q)}.\]
\end{claim}

\subsection{The proof of Lemma~\ref{lem:ratios}}

We proceed in a straightforward manner. 
First we compute formally the total derivative matrix $D\Phi$ of $\Phi$ at a potential $Q \in \QQ$. Then we express $\det D\Phi$ as the product of a non-zero term and a polynomial in the values of the potential. Then we show that this polynomial is not identically zero, and thus there is a choice of $Q$ for which $D\Phi$ is invertible.

\paragraph{Step 1: computing the total derivative.}

Let us compute
\[ (D\Phi)_{i,j} = \partial_j \frac{\lambda_i}{\lambda_{n+1}} =  \frac{ (\partial_j \lambda_i)\lambda_{n+1} - \lambda_i \partial_j \lambda_{n+1}}{\lambda_{n+1}^2}.\]
We can multiply each element of $D\Phi$ by $\lambda_{n+1}$, since that doesn't change whether $\det D\Phi$ is zero. Let us next append an $n+1$st row to this matrix whose $j$th element is $\partial_j \lambda_{n+1}$ and an $n+1$st column in which all elements are 0 except for the last one which is $\lambda_{n+1}$. This still does not change whether the determinant is 0. Finally, for each $i=1,\dots,n$, add  $\lambda_i / \lambda_{n+1}$ times the last row to the $i$th row. This does not change the determinant. 

Now we have arrived at an $(n+1)\times(n+1)$ matrix $\tilde{M}$, whose determinant is zero if and only if $\det D\Phi$ was zero, and whose entries are
\[ \tilde{M}_{i,j} = \left\{ \begin{array}{lcl} \partial_j \lambda_i & : & j\leq n \\ \lambda_i & : & j = n+1  \end{array}\right.\]
Our goal is to show that $\det \tilde{M} \neq 0$. By Claim~\ref{cl:implicit} we have $\partial_j \lambda_i = F_j(\lambda_i, Q)/F'(\lambda_i,Q)$. The denominator is non-zero because of the assumption that $H$ has simple eigenvalues. Thus we can multiply the $i$th row of $\tilde{M}$ by $F'(\lambda_i,Q)$ without changing whether the determinant is 0. Let $M$ denote the matrix obtained this way. Thus
\begin{eq}{eq:mij} M_{i,j} = \left\{ \begin{array}{lcl} F_j(\lambda_i,Q) & : & j\leq n \\ \lambda_i F'(\lambda_i,Q) & : & j = n+1  \end{array}\right.\end{eq}
So it suffices to show that $\det M \neq 0$ for some choice of $Q \in \QQ$. 

\paragraph{Step 2: expressing $\det M$ as a polynomial in $Q$.} 

Let us first consider $\det M$ as a polynomial in the variables $\lambda_1,\dots,\lambda_{n+1},Q_1,\dots,Q_n$. As such, it is clearly alternating in the $\lambda_i$s, so by the fundamental theorem of alternating polynomials, it can be written as a product of a polynomial symmetric in the $\lambda_i$s and $\Lambda = \prod_{i < k} (\lambda_k - \lambda_i)$. By the simple eigenvalue assumption $\Lambda \neq 0$, so it suffices to show that $P(\lambda_1,\dots,\lambda_{n+1},Q_1,\dots,Q_n) = \det M / \Lambda \neq 0$. 

Notice that the various elementary symmetric polynomials of the $\lambda_i$s are exactly the coefficients of the polynomial $F(x,Q)$, hence are themselves polynomials in the $Q_j$s. Substituting the appropriate expressions into $P$ we get a new polynomial $T(Q_1,\dots,Q_n)$ whose value coincides with $P(\lambda_1,\dots,\lambda_{n+1},Q_1,\dots,Q_n) = \det M / \Lambda$.

Thus it suffices to show that $T(Q_1,\dots,Q_n) \neq 0$ for some choice of $Q \in \QQ$. Since $\QQ$ is a non-empty open set, this is equivalent to showing that $T$ is not the identically 0 polynomial. We are going to show this by expressing fairly explicitly the highest degree term in $T$. This will be done in multiple steps. First we compute the top degree parts of $P$ and then analyze what happens after the substitution. 

\paragraph{Step 3: computing the top degree parts of the polynomial $P$.}

Let us start by examining the polynomial $F(x,Q)$. After applying Gaussian elimination to the row and column corresponding to $u$ in the Hamiltonian, we get that 
\begin{eq}{eq:Fdet} F(x,Q) = \det \left[\begin{array}{c|ccccc}
x/2 &&& w^T &&\\\hline &Q_1+x&&\ &&\\&&Q_2+x&&?&\\w&&&\ddots&\\&&?& &Q_{n-1}+x&\\&&&&&Q_{n}+x
\end{array}\right]
\end{eq} where $w$ is a 0-1 vector having 1s exactly at the neighbors of $u$ and $v$. Let $W$ denote the set of neighbors of $u$. So $w_j = 1 \leftrightarrow j \in W$. Let $E$ denote the set of edges of $G$ not incident to $u$ or $v$. The degree of $F$ as an element of $\Q[x,Q_1,\dots,Q_n]$ is $n+1$. 

\begin{definition}
For any polynomial $J$ of (hypothetical) degree $d$ let us denote by $J^c$ its degree $d-c$ homogeneous part. Since we will not consider powers of polynomials, this should not lead to confusion.
\end{definition}

It is clear from \eqref{eq:mij} and \eqref{eq:Fdet} that the degree of $F_j$ is $n$, so $\det M$ has degree $n^2+n+1$. Then $P$ has degree $n^2+n+1-n(n+1)/2 = n(n+1)/2+1$. Using the above notation, it is clear from \eqref{eq:Fdet} that 
\[ F^0(x,Q) = x/2 \prod_{j=1}^n (x+Q_j),\]  $F^1 = 0$, and 
\begin{eq}{eq:f2} F^2(x,Q) =  - \sum_{(ab) \in E(G)} x/2 \prod_{s \neq a,b} (x+Q_s) - \sum_{a \in W} \prod_{s\neq a} (x+Q_s). \end{eq}
Then the entries of $M$ also naturally split according to their homogeneous degrees. However, we are interested in computing the homogeneous parts of $\det M$. For this reason, let us introduce the following matrices: $M^0$ is the matrix consisting of the top degree part of each entry of $M$.
\begin{eq}{eq:m0}
M^0_{i,j} = \left\{ \begin{array}{llll} F^0_j(\lambda_i,Q) &= \lambda_i/2 \prod_{s \neq j} (\lambda_i+Q_s) & : & j\leq n \\ \lambda_i \cdot(F^0)'(\lambda_i,Q) &= \lambda_i^2/2 \sum_j \prod_{s\neq j} (\lambda_i+Q_s) +\lambda_i \prod(\lambda_i+Q_j)& : & j = n+1  \end{array}\right.
\end{eq}
Then clearly $(\det M)^0 = \det( M^0)$. Since $F^1 = 0$, the same will hold for all its derivatives and hence for all entries of $M$. Thus it also holds for the determinant: $(\det M)^1 = 0$. Next we compute $(\det M)^2$. This is obtained by keeping the top degree part from each entry in $M$ except for the entries in a single column, where we replace them by the second highest degree part. So for any $1 \leq k \leq n+1$ we introduce the matrix $M^{2(k)}$ that has entries
\begin{eq}{eq:m2k} 
M^{2(k)}_{i,j} = \left\{ \begin{array}{lll} M^0_{i,j} & : & j \neq k \\ F^2_j(\lambda_i,Q) & : & j = k \leq n \\ \lambda_i \cdot (F^2)'(\lambda_i,Q) & : & j = k =n+1
 \end{array}\right.
 \end{eq}
Using this notation we get, by the multi-linearity of the determinant as a function of columns, that
\begin{eq}{eq:detm2} (\det M)^2 = \sum_{k=1}^{n+1} \det M^{2(k)}.\end{eq} 
It is clear that $\det M^0$, as well as all the $\det M^{2(k)}$s are alternating in the $\lambda_i$s, so they are all divisible by $\Lambda$. Thus we get that the top degree parts of the polynomial $P = \det M /\Lambda$ are 
\begin{eq}{eq:P^c} P^0 = \frac{\det M^0}{\Lambda} \mbox{; } P^1 = 0 \mbox{; and } P^2 = \frac{\sum_k \det M^{2(k)}}{\Lambda}.\end{eq} Both $P^0$ and $P^2$ are symmetric in the $\lambda_i$s.

\paragraph{Step 4: the substitution.}
Finally, let us consider what happens when we substitute the coefficients of $F(x,Q)$ in place of the elementary symmetric polynomials in the $\lambda_i$s. Let us write
\[ F(x,Q) = x^{n+1} + \sum_{k=1}^{n+1} (-1)^k S_k(Q) x^{n+1-k}\] where $S_k \in \Q[Q_1,\dots,Q_n]$. Then, since the $\lambda_i$s are exactly the roots of $F(x,Q)$, we get that  \[ \sigma_k(\lambda_1,\dots,\lambda_{n+1}) = S_k(Q)\] where $\sigma_k$ is the $k$th elementary symmetric polynomial. Furthermore, it follows from a careful but straightforward examination of \eqref{eq:Fdet} that 
\begin{eq}{eq:sk0} S_k^0 = \frac{(-1)^k}{2} \sigma_k(Q_1,\dots,Q_n),\end{eq}
\begin{eq}{eq:sk1} S_k^1 = (-1)^{k+1} \sum_{a \in W} \sigma_{k-1}(Q_1, \dots, Q_{a-1},Q_{a+1},\dots,Q_n).\end{eq}
Note that in particular we have 
\begin{eq}{eq:sn+1} S_{n+1}^0 = S_{n+1}^1 = 0,\end{eq} since the $n+1$st (respectively the $n$th) symmetric polynomial of $n$ (respectively $n-1$) variables is 0.

To analyze the substitution, let us denote the space of polynomials symmetric in the $\lambda_i$s by $\Q^{sym}[\lambda_1,\dots,\lambda_{n+1},Q_1,\dots,Q_n]$, and define the map 
\begin{align*}
\Psi : \Q^{sym}[\lambda_1,\dots,\lambda_{n+1},Q_1,\dots,Q_n] &\to \Q[Q_1,\dots,Q_n] \\
 \sigma_k(\lambda_1,\dots,\lambda_{n+1})\cdot J(Q) &\mapsto S_k \cdot J(Q).
\end{align*}

As $P$ had degree $n(n+1)/2 +1$, we consider $T= \Psi(P)$ as a hypothetical degree $n(n+1)/2+1$ polynomial as well, (where the leading coefficient may be 0). Equation \eqref{eq:m0} shows that $M^0_{i,j}$ is divisible by $\lambda_i$ for all $j$. Hence $\det M^0$ is divisible by $\prod \lambda_i$. Thus $P^0$ is also divisible by $\prod \lambda_i$. Let us write $P^0 = \prod \lambda_i \cdot R$ where $R$ has degree $n(n+1)/2-n$. Then $\Psi(P^0) = \Psi(\prod \lambda_i) \cdot \Psi(R) = S_{n+1} \Psi(R)$.

\begin{claim}
$T^0 = T^1 = 0.$
\end{claim}
\begin{proof}
 From \eqref{eq:sn+1} we see that $S_{n+1}$ has actual degree at most $n-1$, and thus $\Psi(P^0)$ has actual degree at most $n(n+1)/2-1$, so it does not contribute to $T^0$ and $T^1$. On the other hand we have seen that $P^1 = 0$, and for any $c \geq 2$ the actual degree of $\Psi(P^c)$ is also at most $n(n+1)/2 - 1$, so none of these contribute to $T^0$ and $T^1$. The claim follows. 
\end{proof}

\paragraph{Step 5: expressing $T^2$ as a linear combination.} In this step we study the dependence of $T^2 = T^2_G$ on the graph $G$. For any $1 \leq a \leq n$ let $G_a$ be the graph on $u,v,1,2,\dots, n$ that has only two edges: $(ua)$ and $(va)$. 
To this graph we can associate the polynomial $T^2_{G_a}$ analogously to the definition of $T^2_G$. The goal of this step is to show the following.

\begin{lemma}\label{lem:T_lin_comb}
Let $W$ denote the set of neighbors of $u$ in $G$. Then 
\[ T^2_G = \sum_{w\in W} T^2_{G_a}.\]
\end{lemma}

\begin{proof}
Let us introduce
\begin{align*}
\Psi^0 : \Q^{sym}[\lambda_1,\dots,\lambda_{n+1},Q_1,\dots,Q_n] &\to \Q[Q_1,\dots,Q_n] \\
 \sigma_k(\lambda_1,\dots,\lambda_{n+1})\cdot J(Q) &\mapsto S_k^0 \cdot J(Q).
\end{align*}

\begin{claim} The map $\Psi^0$ is independent of the original graph, since by \eqref{eq:sk0} none of the $S_k^0$s depend on the actual graph structure.
\end{claim}

\begin{claim}\label{clm:psi0}
For any polynomial $J \in \Q^{sym}[\lambda_1,\dots,\lambda_{n+1},Q_1,\dots,Q_n]$ we have 
\[ (\Psi(J))^0 = \Psi^0 (J^0).\] 
\end{claim}

\begin{proof}
Since the degree of $S_k$ is equal to the degree of $\sigma_k(\lambda_1,\dots,\lambda_{n+1})$, the map $\Psi$ does not increase the homogeneous degree, and since the degree of $S_k - S_k^0$ is strictly less than $k$, we get that $\Psi - \Psi^0$ strictly decreases the homogeneous degree. Thus the top degree part of $\Psi(J)$ has to coincide with $\Psi^0(J^0)$.
\end{proof}

We can write $P = P^0 + P^2 + \sum_{c > 2} P^c = \prod \lambda_i \cdot R + P^2 + \sum_{c>2} P^c$. Then 
\[T = \Psi(P) = \Psi(\prod \lambda_i \cdot R)  + \Psi(P^2) + \sum_{c>2} \Psi(P^c) = S_{n+1} \Psi(R) + \Psi(P^2) + \sum_{c>2} \Psi(P^c).\] Clearly $\sum_{c>2} \Psi(P^c)$ does not contribute to $T^2$. We have seen that $S^0_{n+1} = S^1_{n+1}= 0$ so the actual degree of $S_{n+1}$ is at most $n-1$, hence by Claim~\ref{clm:psi0} the contribution of $S_{n+1}\Psi(R)$ to $T^2$ is exactly $S^2_{n+1}\Psi^0(R^0)$ and again by Claim~\ref{clm:psi0} the contribution of $\Psi(P^2)$ to $T^2$ is exactly $\Psi^0(P^2)$. Thus we get that 

\begin{eq}{eq:t2} 
T^2 = S^2_{n+1} \Psi^0(R^0) + \Psi^0(P^2)
\end{eq}

Let us first study the first term of this expression. We have seen that $P^0$ doesn't depend on the graph, and hence $R^0$ doesn't either. $\Psi^0$ is also independent of the graph. We know that $S_n$ is equal (up to sign) to the constant term of $F(x,Q)$, that is simply $F(0,Q)$. From \eqref{eq:Fdet} it is easily seen by expanding the determinant that 
\[ S_{n+1}^2 = (-1)^n \sum_{a \in W} \prod_{s \neq a} Q_s.\] 

Next, let us look at the second term of \eqref{eq:t2}. According to \eqref{eq:detm2} \[ P^2 = \frac{(\det M)^2 }{\Lambda} = \frac{\sum_{k=1}^{n+1} \det M^{2(k)}}{\Lambda}.\]  It can be immediately seen from \eqref{eq:m2k} that each entry in the $i$th row of $M^{2(n+1)}$ is divisible by $\lambda_i$, hence $\prod \lambda_i$ divides $\det M^{2(n+1)}$, thus $\Psi^0(\det M^{2(n+1)}/\Lambda) = 0$. 

For $k \leq n$ let us further decompose $M^{2(k)}$ according to \eqref{eq:f2} and \eqref{eq:m2k}.  Let us write $F^{2(a,b)}(x,Q) = - x/2 \prod_{s\neq a,b} (x+Q_s)$ and $F^{2(a)}(x,Q) = -\prod_{s\neq a} (x+Q_s)$. Then \[ F^2(x,Q) = \sum_{(ab) \in E(G)} F^{2(a,b)}(x,Q) + \sum_{a\in W} F^{2(a)}(x,Q).\] Define  $M^{2(k)(a,b)}$ to be the matrix that coincides with $M^0$ in all columns except the $k$th, and whose entries in the $k$th column are $F_k^{2(a,b)}(\lambda_i, Q)$. Similarly let $M^{2(k)(a)}$ be the matrix that coincides with $M^0$ in all but the $k$th column, where the entries are $F_k^{2(a)}(\lambda_i,Q)$. Thus
\[ M^{2(k)} = \sum_{(a,b) \in E(G)} M^{2(k)(a,b)} + \sum_{a \in W} M^{2(k)(a)},\] and since these matrices only differ in their $k$th column, we get 
\[ \det M^{2(k)} = \sum_{(a,b) \in E(G)} \det M^{2(k)(a,b)} + \sum_{a \in W} \det M^{2(k)(a)},\] and thus 
\[ \Psi^0\left( \frac{\det M^{2(k)}}{\Lambda}\right) = \sum_{(a,b) \in E(G)} \Psi^0\left(\frac{\det M^{2(k)(a,b)}}{\Lambda}\right) + \sum_{a \in W} \Psi^0\left(\frac{\det M^{2(k)(a)}}{\Lambda}\right).\] Note, however, that since $F^{2(a,b)}(x,Q)$ is divisible by $x$, so is $F^{2(a,b)}_j(x,Q)$, so each entry in the $i$th row of $M^{2(k)(a,b)}$ is divisible by $\lambda_i$ for $k \leq n$. This means that $\det M^{2(k)(a,b)}$ is divisible by $\prod \lambda_i$, hence $\Psi^0(\det M^{2(k)(a,b)}/\Lambda) = 0$. 

Putting together everything we get that 
\begin{multline}\label{eq:t2final}T^2 = (-1)^n \sum_{a \in W} \prod_{s\neq a}Q_s \cdot \Psi^0(R^0) + \sum_{k=1}^n \sum_{a\in W} \Psi^0\left(\frac{\det M^{2(k)(a)}}{\Lambda}\right)  =\\ = \sum_{a\in W}\left( \sum_{k=1}^n \Psi^0\left(\frac{\det M^{2(k)(a)}}{\Lambda}\right) + (-1)^n \prod_{s\neq a} Q_s \cdot \Psi^0(R^0)\right)\end{multline}

We have seen that $\Psi^0$ and $R^0$ are independent of the graph. From the explicit form of $F^{2(a)}(x,Q) = - \prod_{s\neq a}(x+Q_s)$ we see that it is also independent of the graph. Then the same follows for the polynomial $\det M^{2(k)(a)}$. This implies that 
\[ T^2_{G_a} = \sum_{k=1}^n \Psi^0\left(\frac{\det M^{2(k)(a)}}{\Lambda}\right) + (-1)^n \prod_{s\neq a} Q_s \cdot \Psi^0(R^0),\] and thus 
\begin{eq}{eq:t2sum} T^2 = T^2_G = \sum_{a \in W} T^2_{G_a}.\end{eq}
\end{proof}

\paragraph{Step 6: computing $T^2_{G_a}$.} The proof of Lemma~\ref{lem:ratios} will be complete once we compute $T^2_{G_a}$ and show that the sum in \eqref{eq:t2sum} cannot be zero. 

However, computing $T^2_{G_a}$ is easy since we can explicitly compute the dependence of the $\lambda_i$s on the $Q_i$s. Let us now use $M$ to denote $M_{G_a}$. This dependence is rather simple and allows for a direct computation of $M_{i,j}$, and through that $T_{G_a}$.  We will focus on $a=n$ without loss of generality.

From \eqref{eq:Fdet} we see that $F_{G_1}(x,Q) = 1/2\cdot(x^2 + Q_n x -2)\prod_{i=1}^{n-1} (x+Q_i)$. The roots of this polynomial are $\lambda_i = -Q_i$ for $ i = 1,\dots,n-1$, and $\lambda_{n/n+1} = (-Q_n \pm \sqrt{Q_n^2+8})/2$.
Thus for $j < n$ 
\[ M_{i,j} = F_j(\lambda_i,Q) = \left\{
\begin{array}{lll} 
0 &:& i \neq j \\
1/2 (Q_i^2 - Q_n Q_i - 2) \prod_{s\neq i} (Q_s - Q_i) &:& i = j
\end{array} \right. 
\]
Similarly
\begin{align*}
M_{n,n} = F_n(\lambda_n,Q) =&1/2\cdot \lambda_n \prod_{i=1}^{n-1} (\lambda_n+Q_i) \\
M_{n+1,n} = F_n(\lambda_{n+1},Q) = &1/2\cdot \lambda_{n+1} \prod_{i=1}^{n-1}(\lambda_{n+1}+Q_i) \\
M_{n,n+1} = \lambda_n F'(\lambda_n,Q) =& 1/2\cdot\lambda_n (2\lambda_n+Q_n) \prod_{i=1}^{n-1} (\lambda_n+Q_i)\\
M_{n+1,n+1} = \lambda_{n+1} F'(\lambda_{n+1},Q) =&1/2\cdot \lambda_{n+1}(2\lambda_{n+1}+Q_n) \prod_{i=1}^{n-1} (\lambda_{n+1}+Q_i)\\
M_{n,n}M_{n+1,n+1} - M_{n,n+1} M_{n+1,n} =& 1/2\cdot \lambda_n \lambda_{n+1}(\lambda_{n+1} - \lambda_{n})\prod_{i=1}^{n-1} (\lambda_n+Q_i)(\lambda_{n+1}+Q_i) 
\end{align*}
A simple computation, using liberally that $\lambda_i = -Q_i$ for $i\leq n-1$,  then yields
\begin{multline*} \det M = \prod_{i=1}^{n-1} F_i(\lambda_i,Q) \cdot \left(M_{n,n}M_{n+1,n+1} - M_{n,n+1} M_{n+1,n}\right)= \\= \frac{\lambda_n \lambda_{n+1}(\lambda_{n+1}-\lambda_{n})}{2^{n+1}}\prod_{i=1}^{n-1} (Q_i^2-Q_nQ_i -2) \prod_{1 \leq i < j < n} (\lambda_j-\lambda_i)(Q_j-Q_i) \prod_{i=1}^{n-1}(\lambda_n - \lambda_i)(\lambda_{n+1}-\lambda_i) =\\= \Lambda \frac{\lambda_n \lambda_{n+1}}{2^{n+1}}\prod_{i=1}^{n-1} (Q_i^2-Q_nQ_i -2) \prod_{1\leq i < j < n}(Q_j - Q_i).
\end{multline*}
Thus, since $\lambda_n \lambda_{n+1} = 2$, we get
\[ T = \Psi\left(\frac{\det M}{\Lambda}\right) = \frac{1}{2^n} \prod_{i=1}^{n-1} (Q_i^2-Q_i Q_n -2) \prod_{1\leq i < j < n}(Q_j-Q_i)\]
This polynomial should have degree $n(n+1)/2 +1$ but as we have seen, its actual degree is $n(n+1)/2 -1$, so 
\[ T^2_{G_a} = \frac{\pm 1}{2^n} \frac{1}{Q_a} \prod_{i=1}^n Q_i \prod_{1\leq i < j \leq n} (Q_j-Q_i).\]
Hence \[ T^2_G = \sum_{a\in W} \frac{\pm 1}{Q_a}  \prod_{i=1}^n Q_i \prod_{1 \leq i < j \leq n}(Q_j-Q_i),\] and this is clearly not 0. This completes the proof of Lemma~\ref{lem:ratios}.\qed

\section{Graph Products}
Let $G_1\Box G_2$ denote the cartesian product of graphs $G_1$ and $G_2$, that is, $V(G_1\Box G_2) = V(G_1)\times V(G_2)$ and $E(G_1\Box G_2) = \{ \{(u,v),(x,y)\} : u=x\ and\ u\sim y \ or\ v=y\ and\ u\sim x\}$.  It is well known that the adjacency matrix for the cartesian product is given by 
\[
A(G_1\Box G_2) = A(G_1)\otimes I + I \otimes A(G_2)
\]
where $\otimes$ denotes the Kronecker product of matrices. A well-known fact about Kronecker products that will be of use to us is that
\begin{eq}{eq:tensor}
(A\otimes B)(C\otimes D) = AC \otimes BD
\end{eq}
for any matrices $A,B,C,D$ for which the products are defined.  An immediate consequence is the well-known fact that if $\lambda$ is an eigenvalue of $A(G_1)$ with eigenvector $\phi$ and $\mu$ and eigenvalue of $A(G_2)$ with eigenvector $\psi$, then $\lambda+\mu$ is an eigenvalue of $A(G_1\Box G_2)$ with eigenvector $\phi\otimes \psi$.  

In \cite{godsil}, it is shown that if perfect state transfer without potential occurs on $G_1$ and $G_2$ at the same time, then it occurs for the product.  In this section, we will show that the same holds in the presence of a potential.  Our proof is essentially the same as in \cite{godsil}, the only thing that needing to be decided is how to define the potential on the product.

\begin{theorem}\label{thm:product}
Let $G_1,G_2$ be graphs with potentials and $Q_1,Q_2$ respectively.  Then if perfect state transfer occurs from $u$ to $v$ at time $t$ in $G_1$, and from $x$ to $y$ at the same time $t$ in $G_2$, the we have perfect tunneling from $(u,x)$ to $(v,y)$ at the same time $t$ in the graph product $G_1\Box G_2$ with the potential $Q$ given by $Q((u,x),(u,x)) = Q_1(u,u) + Q_2(x,x)$. 
\end{theorem}
\begin{proof}
Let $H_1 = A(G_1) - Q_1$ and $H_2 = A(G_2) - Q_2$ be the Hamiltonians, and denote by $H = A(G_1\Box G_2) - Q$ the Hamiltonian for the product, with $Q$ defined as in the statement of the theorem.  Note that it is clear that $Q = Q_1\otimes I + I \otimes Q_2$, and we observed above that $A(G_1\Box G_2) = A(G_1)\otimes I + I\otimes A(G_2)$, so that $H = H_1\otimes I + I \otimes H_2$.  Thus, the eigenvalues and eigenvectors of $H$ are given by the eigenvalues and eigenvectors of $H_1$ and $H_2$ in the same way as the adjacency matrix.  Define $U_1(t) = e^{itH_1}$, $U_2(t) = e^{itH_2}$, and $U(t) = e^{itH}$.
\begin{lemma}\label{lem:prod}
$U(t) = U_1(t) \otimes U_2(t)$.
\end{lemma}
\begin{proof}
As observed above, the eigenvalues of $H$ are all numbers of the form $\lambda+\mu$ where $\lambda$ ranges over the eigenvalues of $H_1$ and $\mu$ the eigenvalues of $H_2$.  Let $\phi_\lambda$ denote the eigenvector of $H_1$ for $\lambda$, and $\psi_\mu$ the eigenvector of $H_2$ for $\mu$.  Then $\phi_\lambda\otimes \psi_\mu$ is the eigenvector of $H$ for $\lambda+\mu$. Then by properties of Kronecker products, we have
\begin{align*}
U(t) &= \sum_{\lambda,\mu}e^{it(\lambda+\mu)}\phi_\lambda\otimes \psi_\mu \\
&= \sum_{\lambda,\mu} \left(e^{it\lambda}\phi_\lambda\right)\otimes \left(e^{it\mu}\psi_\mu\right)\\
&= \left(\sum_\lambda e^{it\lambda}\phi_\lambda\right)\otimes\left(\sum_\mu e^{it\mu}\psi_\mu\right)\\
&= U_1(t) \otimes U_2(t)
\end{align*}
which gives the lemma.
\end{proof}
With this, we can finish the proof of the theorem.  Since we are assuming perfect state transfer from $u$ to $v$ at time $t$ in $G_1$, and from $x$ to $y$ in $G_2$ at the same time, we have
\begin{align*}
U_1(t)\mathbf 1_u &= \gamma_1\mathbf 1_v\\
U_2(t)\mathbf 1_x &= \gamma_2\mathbf 1_y
\end{align*}
where $|\gamma_i|=1$.  Finally, it is clear that $\mathbf 1_{(u,x)} = \mathbf 1_u \otimes \mathbf 1_x$.  Thus, letting $\gamma = \gamma_1\gamma_2$ (so $|\gamma|=1$), we see from Lemma \ref{lem:prod} and (\ref{eq:tensor}) that
\begin{align*}
U(t)\mathbf 1_{(u,v)} &= \left(U_1(t)\otimes U_2(t)\right)\left(\mathbf 1_u\otimes \mathbf 1_x\right)\\
&= U_1(t)\mathbf 1_u \otimes U_2(t)\mathbf 1_x\\
&= \gamma_1\mathbf 1_v \otimes \gamma_2\mathbf 1_y\\
&= \gamma \mathbf 1_{(v,y)}.
\end{align*}
This completes the proof of the theorem.
\end{proof}

In the previous section, we saw that perfect state transfer occurs in graph which have two vertices that share identical neighborhoods.  This condition is somewhat restrictive, and it is natural to ask if we can produce other examples that do not satisfy this restriction.  Indeed we can, using Theorem \ref{thm:product}, if we take the cartesian product of a graph (with its potential) with itself, then it is possible that the two vertices between which tunneling occurs do not share a neighborhood.  Indeed, taking the product of $P_3$ with itself any number of times gives such an example.

\bibliographystyle{plain}
\bibliography{quantum}

\end{document}